\documentclass[11pt, oneside]{amsart}

\usepackage[english]{babel} 
\usepackage{graphics,color,pgf}
\usepackage{epsfig}
\usepackage[ansinew]{inputenc}
\usepackage[all]{xy}
\usepackage{hyperref}
\newdir{ >}{!/8pt/@{}*@{>}}
\usepackage{amssymb, amsmath,amsthm, mathtools, amscd}
\usepackage{mathrsfs}
\usepackage[margin=1.5in]{geometry}

\makeatletter
\newtheorem*{rep@theorem}{\rep@title}
\newcommand{\newreptheorem}[2]{%
\newenvironment{rep#1}[1]{%
 \def\rep@title{#2 \ref{##1}}%
 \begin{rep@theorem}}%
 {\end{rep@theorem}}}
\makeatother

\newtheorem{lemma}{Lemma}[section]
\newtheorem{thm}[lemma]{Theorem} 
\newtheorem{prop}[lemma]{Proposition}
\newtheorem{cor}[lemma]{Corollary}

\theoremstyle{definition}
\newtheorem{defn}[lemma]{Definition}
\newtheorem{es}[lemma]{Example}

\theoremstyle{remark}
\newtheorem{oss}[lemma]{Remark}

\newtheoremstyle{TheoremNum}
        {\topsep}{\topsep}              %%% space between body and thm
        {\itshape}                      %%% Thm body font
        {}                              %%% Indent amount (empty = no indent)
        {}                     %%% Thm head font
        {.}                             %%% Punctuation after thm head
        { }                             %%% Space after thm head
        {\thmname{\bfseries #1}\thmnote{ \bfseries #3}}%%% Thm head spec
    \theoremstyle{TheoremNum}

\newcommand\calS{{\mathcal S}}

\newcommand\calT{{\mathcal T}}

\newcommand\calB{{\mathcal B}}
\newcommand\calH{{\mathcal H}}

\newcommand\calR{\mathcal{R}}
\newcommand\calJ{\mathcal{J}}

\newcommand\calX{\mathcal{X}}
\newcommand\calD{{\mathcal D}}
\newcommand\calO{{\mathcal O}}

\newcommand\upC{{\textup{C}}}

\newcommand\upH{{\textup{H}}}
\newcommand\upI{{\textup{I}}}
\newcommand\upL{{\textup{L}}}

\newcommand\upT{{\textup{T}}}

\newcommand\bbC{{\mathbb{C}}}

\newcommand\bbF{{\mathbb{F}}}

\newcommand\bbH{{\mathbb{H}}}

\newcommand\bbN{{\mathbb{N}}}

\newcommand\bbR{{\mathbb{R}}}

\newcommand\bfG{{\mathbf{G}}}
\newcommand\bfH{{\mathbf{H}}}

\newcommand\bfL{{\mathbf{L}}}

\newcommand\bfQ{{\mathbf{Q}}}

\newcommand{\pupq}{\textup{PU}(p,q)}
\newcommand{\supq}{\textup{SU}(p,q)}

\newcommand{\puno}{\textup{PU}(n,1)}

\begin{document}

\title[Orbital cohomology and K\"{a}hler rigidity]{Orbital cohomology and K\"{a}hler rigidity}

\author[A. Savini]{A. Savini}
\address{Section de Math\'ematiques, University of Geneva, Rue Du Conseil G\'eneral 7-9, Geneva 1205, Switzerland}
\email{alessio.savini@unige.ch}

\date{\today.\ \copyright{A. Savini 2022}.
The author is supported by the SNF grant n. 20020-192216.}

\begin{abstract}
In the late $70$'s Feldman and Moore \cite{feldman:moore} defined the cohomology associated to a countable equivalence relation with coefficients in an Abelian Polish group. When the equivalence relation is the orbital one, that is it is induced by a measure preserving action of a countable group $\Gamma$ on a standard Borel probability space $(X,\mu)$, it still makes sense to consider the Feldmann-Moore $1$-cohomology with $G$-coefficients, where this time $G$ can be any topological group. The latter cohomology, denoted by $\upH^1(\Gamma \curvearrowright X;G)$, is very misterious and hard to compute, except for some exceptional cases.                      

In this expository paper we are going to focus our attention on the particular case when $\Gamma$ is a finitely generated group and $G$ is a Hermitian Lie group. We are going to give some recent rigidity results in this context and we will see how those results can be used to say something relevant about (some subsets of) the orbital cohomology. 
\end{abstract}
  
\maketitle

\section{Introduction}

In Dynamics an interesting and fruitul topic of research is \emph{measured group theory}. Given a measure preserving action of a finitely generated group $\Gamma$ on a standard Borel probability space $(X,\mu)$, measured group theory studies the interplay between the algebraic properties of the group $\Gamma$ and the dynamical properties (for instance the structure of orbits) of the $\Gamma$-action on $(X,\mu)$. 

One of the most celebrated result in this field is the orbit equivalence rigidity theorem by Zimmer \cite[Theorem 4.3]{zimmer:annals}. Roughly speaking, two finitely generated groups $\Gamma, \Lambda$ acting in an essentially free and measure preserving way on two standard Borel probability spaces $(X,\mu)$ and $(Y,\nu)$, respectively, are \emph{orbit equivalent} if there exists a Borel isomorphism $\varphi:X \rightarrow Y$ sending $\Gamma$-orbits to $\Lambda$-orbits. More precisely, we require that the Borel isomorphism $\varphi$ respects the involved measures, that is the direct image of $\mu$ is $\nu$, and $\varphi(\Gamma.x) = \Lambda.\varphi(x)$, for almost every $x \in X$. When $\Gamma$ and $\Lambda$ are two lattices contained in two higher rank center free simple Lie groups $G,H$, respectively, Zimmer proved that if the actions $\Gamma \curvearrowright (X,\mu)$ and $\Lambda \curvearrowright (Y,\nu)$ are orbit equivalent, then $G$ and $H$ must be isomorphic. Such a rigidity phenomenon is in sharp contrast with what happens in the case of amenable groups, for example. In fact Ornstein and Weiss \cite{OW80} proved that any two ergodic measure preserving actions of two infinite countable amenable groups must be orbit equivalent. 

We denote by $\calR_\Gamma$ the equivalence relation such that two points of $(X,\mu)$ are related if and only if they are in the same $\Gamma$-orbit, and we adopt the analogous notation $\calR_\Lambda$ for the $\Lambda$-action on $(Y,\nu)$. One easily sees that the definition of orbit equivalence can be naturally rewritten in terms of the associated orbital equivalence relations. This is an easy case of the more general idea of translating the study of measure preserving actions of countable groups in terms of their orbital equivalence relations. The latter idea inspired the theory of \emph{measured equivalence relations}, that is the study of the structural properties of a \emph{countable} equivalence relation (\emph{i.e.} with countable equivalence classes) defined over a probability space $(X,\mu)$. An important contribution to this topic was given in the late $70$'s by Feldman and Moore \cite{moore1976,feldman:moore}. They introduced the cohomology $\upH^\bullet(\calR;T)$ of a measured equivalence relation $\calR$ with coefficients in an Abelian Polish group $T$. Although Polish groups are required to give a consistent definition of higher order cohomology, one can consider the $1$-cohomology $\upH^1(\calR;G)$ with coefficients in $G$, where $G$ is any topological group. In this context a \emph{cocycle} is a Borel measurable map $c:\calR \rightarrow G$ satisfying the relation $c(x,z)=c(y,z)c(x,y)$ for almost every pair $(x,y),(y,z),(x,z) \in \calR$. In the same spirit, two cocycles $c_1,c_2$ are \emph{cohomologous} if there exists a Borel measurable map $f:X \rightarrow G$ such that $f(y)c_1(x,y)=c_2(x,y)f(x)$ for almost every $(x,y) \in \calR$. 

When $\calR = \calR_\Gamma$ is an orbital equivalence relation, the understanding of its $1$-cohomology $\upH^1(\Gamma \curvearrowright X;G):=\upH^1(\calR_\Gamma;G)$ has attracted the interest of many Mathematicians so far. The study of this exotic cohomology theory in full generality may reveal quite harsh. For this reason, it could be helpful to restrict the attention to specific families of groups, for both $\Gamma$ and $G$. For instance, when $G$ is algebraic, it makes sense to refer to the subset $\upH^1_{ZD}(\Gamma \curvearrowright X;G)$ of \emph{Zariski dense} cohomology classes, whose study can be easier. When $\Gamma$ is an irreducible higher rank lattice and $G$ is an algebraic group over a local field, Zimmer superrigidity theorem \cite{zimmer:annals} ensures that every Zariski dense cohomology class contains a (Zariski dense) representation as representative. Equivalently, we have a surjection from the space $\textup{Rep}_{ZD}(\Gamma;G)$ of Zariski dense representations modulo $G$-conjugation to the Zariski dense orbital cohomology $\upH^1_{ZD}(\Gamma \curvearrowright X;G)$.

In this short expository paper we will focus our attention on the particular case when $G$ is a \emph{Hermitian} Lie group. We say that $G$ is Hermitian if the associated symmetric space $\calX$ admits a $G$-invariant complex structure compatible with its Riemannian metric. Additionally, we call $G$ \emph{of tube type} if $\calX$ can be biholomorphically realized as $V+i \Omega$, where $V$ is a real vector space and $\Omega \subset V$ is a proper convex cone. 

Let $\Gamma$ be a finitely generated group, $(X,\mu)$ be an ergodic standard Borel probability $\Gamma$-space and consider a simple Hermitian Lie group $G$ not of tube type. In this setting a measurable cocycle boils down to a measurable map $\sigma:\Gamma \times X \rightarrow G$ such that $\sigma(\gamma_1\gamma_2,x)=\sigma(\gamma_1,\gamma_2.x)\sigma(\gamma_2,x)$ for every $\gamma_1,\gamma_2 \in \Gamma$ and for almost every $x \in X$. Since $G$ is Hermitian, the symmetric space $\calX$ admits a closed differential $2$-form $\omega_{\calX}$, called \emph{K\"{a}hler form}, which induces a class $\kappa^b_G$ in the second bounded cohomology group $\upH^2_{cb}(G;\bbR)$ and generates it. Exploiting such a class we can define its pullback $\upH^2_b(\sigma)(k^b_G)$ along any measurable cocycle $\sigma$ and the pullback will lie in the bounded cohomology group $\upH^2_b(\Gamma;\upL^\infty(X,\bbR))$. The main theorem in this context is that the pullback class is a complete invariant of a Zariski dense cocycles (actually of its cohomology class). In this way, we obtain an injection of $\upH^1_{ZD}(\Gamma \curvearrowright X;G)$ into $\upH^2_b(\Gamma;\upL^\infty(X;\bbR))$ whose image avoids the trivial class. The latter result, obtained in collaboration with Sarti \cite{sarti:savini3}, is a generalization of a previous theorem by Burger, Iozzi and Wienhard \cite{BI04,BIW07} for Zariski dense representations. Such generalization allows us to show that $\upH^1_{ZD}(\Gamma \curvearrowright X;G)$ is empty for some lattices satisfying a suitable cohomological condition. 

When $\Gamma<\puno$, where $n \geq 2$, is a lattice and $G=\pupq$, for $1 \leq p \leq q$, something more can be said. Using the pullback class $\upH^2_b(\sigma)(\kappa^b_G)$ we can introduce a numerical invariant, called \emph{Toledo invariant}, for (the cohomology class of) a measurable cocycle $\sigma$. Such invariant has bounded absolute value, so we are allowed to define \emph{maximal cocycles} as those ones attaining the maximum. We will see that maximal Zariski dense cocycles are \emph{superrigid}, that is they admit a representation as representative \cite[Theorem 2]{sarti:savini1}.  Moreover, applying a previous result by Pozzetti \cite{Pozzetti}, we immediately see that each representation lying in $\upH^1_{ZD}(\Gamma \curvearrowright X;G)$ comes actually from a representation of the ambient group $\puno$. As a consequence, the set $\upH^1_{\max,ZD}(\Gamma \curvearrowright X;G)$ must be empty whenever $1<p<q$, generalizing a result given by Pozzetti for representations. 

\subsection*{Plan of the paper} Section \ref{sec herm space} is devoted to the main definitions and results about Hermitian symmetric spaces. We will quickly review the notion of tupe type domains, Shilov boundary, Bergmann kernels and Hermitian triple product. Then we move to Section \ref{sec orbit rel} where we introduce the orbital cohomology. In Section \ref{sec bound cohom} we recall the bounded K\"{a}hler class and in Section \ref{sec pullback class} we remind its pullback along a measurable cocycle. We conclude with Section \ref{sec rigidity} and \ref{sec maximal}, where we report a list of the main results we have in this context. 

\subsection*{Acknowledgements} I would like to thank Andrea Seppi and the University of Grenoble for the invitation to the TSG seminars and Andrea Seppi for having proposed me to write this paper. 

\section{Hermitian symmetric spaces}\label{sec herm space}

In this section we are going to introduce the main definitions and results about Hermitian symmetric spaces. For more details about this topic we refer the reader either to the papers by Burger, Iozzi and Wienhard \cite{BI04,BIW07} or to the book chapter by Koranyi \cite{Kor00}. 

Before starting, recall that a group $\bfG$ is called \emph{algebraic} over $\bbR$ if it can be realized as the zero set of a (finite) family of $\bbR$-polynomials and both the multiplication and the inversion in $\bfG$ are $\bbR$-algebraic maps. Given a real algebraic group, we can restrict ourselves to the \emph{real points} of $\bfG$, namely the subset $\bfG(\bbR)$ of the real solutions satisfying the polynomial equations which define $\bfG$. Finally, we will denote by $\bfG(\bbR)^\circ$ the connected component of the neutral element of $\bfG(\bbR)$. 

\begin{defn}\label{def hermitian group}
A symmetric space $\calX$ associated to a connected semisimple Lie group $G$ is \emph{Hermitian} if it admits a $G$-invariant complex structure $\calJ_{\calX}$ compatible with its Riemannian tensor. If $\bfG$ is a connected adjoint semisimple $\bbR$-algebraic group, we say that the group $G=\bfG(\bbR)^\circ$ is \emph{Hermitian} (or of \emph{Hermitian type}) if the associated symmetric space is Hermitian. 
\end{defn}

The first example of Hermitian Lie group to keep in mind is given by $G:=\supq$, namely the subgroup of $\mathrm{SL}(p+q,\bbC)$ whose elements are matrices preserving the Hermitian form $h_{p,q}$ with signature $(p,q)$. If we set $d=\min\{p,q\}$, the symmetric space $\calX_{p,q}$ associated to $\supq$ parametrizes the $d$-dimensional linear subspaces of $\bbC^{p+q}$ whose restriction of $h_{p,q}$ is positive definite. 

A Hermitian symmetric space $\calX$ is called \emph{of tube type} if it can be biholomorphically realized as $V+i\Omega$, where $V$ is a real vector space and $\Omega \subset V$ is a proper convex cone. When such realization cannot be done, we say that $\calX$ is not of tube type. Going back to our example $\calX_{p,q}$, one can see that the latter is of tube type if and only if $p=q$. In this case $\calX_{p,p}$ is biholomorphic to $\textup{Herm}(p,\bbC) + i \textup{Herm}^+(p,\bbC)$, where $\textup{Herm}(p,\bbC)$ is the space of Hermitian matrices and $\textup{Herm}^+(p,\bbC)$ is the cone of positive definite ones. It is worth noticing that for $p=q=1$, the symmetric space $\calX_{1,1}$ boils down to upper-half plane realization of the hyperbolic plane $\bbH^2_{\bbR}$.

For any Hermitian symmetric space $\calX$ there always exists a bounded domain $\calD_{\calX}$ of some finite dimensional complex space $\bbC^n$ such that $\calX$ and $\calD_{\calX}$ are biholomorphic. The domain $\calD_{\calX}$ is usually called \emph{bounded realization} (or \emph{Harish-Chandra realization}) of $\calX$ (see \cite[Theorem III.2.6]{Kor00} for more details). The group $G$ of holomorphic isometries of $\calX$ acts via biholomorphisms on its bounded realization $\calD_{\calX}$. Furthermore such action can be continuously extended to the topological boundary $\partial \calD_{\calX}$. In general the latter is not a homogeneous $G$-space, but it admits a unique closed $G$-orbit called Shilov boundary. Here we will introduce the Shilov boundary starting from its analytic interpretation. 

\begin{defn}\label{def shilov boundary}
Let $\calD \subset \bbC^n$ be a bounded domain. The \emph{Shilov boundary} of $\calD$ is the unique minimal closed subset $\calS_{\calD}$ of $\partial \calD$ such that, for any continuous function $f$ on the closure $\overline{\calD}$ and homolorphic in the interior $\calD$, we have that
$$
|f(z)| \leq \max_{y \in \calS_{\calD}}|f(y)|,
$$
for every $z \in \calD$. 
\end{defn}

The previous definition can be restated by saying that $\calS_{\calD}$ is the unique minimal closed subset to add to $\calD$ so that the maximum principle can be applied for a homolorphic function which is continuous on the closure $\overline{\calD}$. 

In the particular case when $\calD=\calD_{\calX}$ is the bounded realization of a Hermitian symmetric space $\calX$, the Shilov boundary $\calS_{\calX}$ is a homogeneous $G$-space, being the unique closed $G$-orbit of a given point \cite[Section 2.3]{BIW07}. To keep track of our favourite example, when $G=\supq$, the Shilov boundary $\calS_{p,q}$ parametrizes all the possible $d$-dimensional linear subspaces of $\bbC^{p+q}$ which are totally isotropic with respect to $h_{p,q}$. Notice that the topological boundary $\partial \calD_{p,q}$ parametrizes the space on which $h_{p,q}$ is semi-definite, thus $\calS_{p,q}$ is a proper subset of the topological boundary. The $G$-homogeneity of $\calS_{p,q}$ is due to the fact that it can be realized as the quotient $G/Q$, where $Q$ is the stabilizer of a fixed totally isotropic subspace with maximal dimension $d$ (say the space generated by the first $d$-vectors $\langle e_1,\ldots,e_d \rangle$ of the canonical basis). This identification is not accidental and can be generalized. More precisely, let $\bfG$ be a connected adjoint semisimple $\bbR$-algebraic group obtained by complexifying a Lie group of Hermitian type $G=\bfG(\bbR)^\circ$. Burger, Iozzi and Wienhard \cite[Section 2.3.1]{BIW07} proved that there exists a proper \emph{maximal parabolic subgroup} $\bfQ<\bfG$, such that $\calS_{\calX}$ corresponds to the real points of the algebraic variety $\bfG/\bfQ$. More precisely $\calS_{\calX}$ is isomorphic to the quotient $(\bfG/\bfQ)(\bbR)=G/Q$, where $Q=G \cap \bfQ$. Also in the product $\calS_{\calX} \times \calS_{\calX}$ we can find a unique open $G$-orbit, denoted by $\calS_{\calX}^{(2)}$, whose elements are pairs of \emph{transverse} points. In the case of $G=\supq$, the subset of transverse pairs in $\calS_{p,q}^{(2)}$ is precisely the subset of pairs of linear subspaces $(V,W)$ which are \emph{linearly transverse}, that is $V \cap W=\{ 0 \}$. 

Let $g_{\calX}$ the Riemannian tensor of the symmetric space $\calD_{\calX}$ and let $\calJ_{\calX}$ the $G$-invariant complex structure. If we define 
$$
(\omega_{\calX})_a(X,Y):=(g_{\calX})_a(X,(\calJ_{\calX})_a(Y)) , 
$$
for every $X,Y \in T_a\calD_{\calX}$, we obtain a differential $2$-form $\omega_{\calX}$ called \emph{K\"{a}hler form}. The latter is clearly $G$-invariant and hence closed by Cartan's lemma \cite[VII.4]{Hel01}. As a consequence, we can consider, for any triple of points $x,y,z \in \calD_{\calX}$, the integral
$$
\beta_{Berg}(x,y,z):=\int_{\Delta(x,y,z)} \omega_{\calX},
$$
where $\Delta(x,y,z)$ is any smooth triangle with geodesic sides and vertices $x,y,z$. The closedness of $\omega_{\calX}$ guarantees that $\beta_{Berg}$ does not depend on the choice of the particular filling triangle $\Delta(x,y,z)$. One of the most important properties of $\beta_{Berg}$ is that it encodes information about the complex and analytic structure of the domain $\calD_{\calX}$. In fact the following equation holds
\begin{equation}\label{eq cocycle kernel}
\beta_{berg}(x,y,z)=-(\arg k_{\calX}(x,y)+\arg k_{\calX} (y,z) + \arg k_{\calX} (z,x))\ , 
\end{equation}
where $\arg$ is the branch of the argument with values in $(-\pi,\pi]$ and $k_{\calX}(\cdot,\cdot)$ is the \emph{Bergman kernel}. The latter is defined as follows: Consider the space of square integrable holomorphic functions $\calH^2(\calD_{\calX})$, namely the space of complex-valued holomorphic functions on $\calD_{\calX}$ whose norm is square integrable with respect to the Lebesgue measure. We have that $\calH^2(\calD_{\calX})$ is a Hilbert space where the evaluation on a point $w \in \calD_{\calX}$ is a bounded linear functional (since $\calD_{\calX}$ is bounded). As a consequence, we can write $f(w)=(f|K_w)$, for some $K_w \in \calH^2(\calD_{\calX})$, where $(\cdot|\cdot)$ is the Hilbert product. The function $k_{\calX}$ is then defined simply by $k_{\calX}(z,w)=(K_z|K_w)$. 

We denote by $\calS_{\calX}^{(3)}$ the set of triples of points that are pairwise transverse. The existence of a continuous extension of $k_{\calX}$ to pairs of transverse points in $\calS_{\calX}$, allows us to extend $\beta_{Berg}$ to $(\calS_{\calX})^{(3)}$. One can see that such extension, still denoted by $\beta_{Berg}$, is a continuous $G$-invariant alternating cocycle in the sense of Alexander-Spanier. Moreover, we have 
$$
\sup_{\calS_{\calX}^{(3)}}|\beta_{Berg}(\eta_0,\eta_1,\eta_2)|=\pi\mathrm{rk}\calX,
$$
where $\mathrm{rk}\calX$ is the real rank of $\calX$ (that is the maximal dimension of a flat in $\calX$). The restriction of $\beta_{Berg}|_{(\calS_{\calX})^{(3)}}$ to triples of points that are pairwise transverse can be further extended to the whole product $(\calS_{\calX})^3$ and such extension, denoted by $\beta_{\calX}$, is measurable and satisfies the same properties of $\beta_{Berg}$. 

We conclude this introduction about Hermitian symmetric spaces by talking about the \emph{Hermitian triple product}. Exploiting Bergman kernels, we can define
$$
\langle \cdot, \cdot,\cdot \rangle: \calS_{\calX}^{(3)} \rightarrow \bbC^\ast, 
$$
$$
\langle \eta_0,\eta_1,\eta_2 \rangle:=k_{\calX}(\eta_0,\eta_1)k_{\calX}(\eta_1,\eta_2)k_{\calX}(\eta_2,\eta_0). 
$$
By \cite[Proposition 2.12]{BIW07} the previous function is continuous and by Equation \eqref{eq cocycle kernel} we have that
\begin{equation}\label{eq Bergman triple product}
\langle \eta_0,\eta_1,\eta_2 \rangle=e^{i \beta_{\calX}(\eta_0,\eta_1,\eta_2)} \mod \bbR^\ast ,
\end{equation}
where $\mod \bbR^\ast$ means that the two terms in the equation above differ by a non-zero real number. By composing $\langle \cdot,\cdot,\cdot \rangle$ with the projection $\bbR^\ast \backslash \bbC^\ast$, where $\bbR^\ast$ acts on $\bbC^\ast$ via dilations, we obtain the \emph{Hermitian triple product}
$$
\langle \langle \cdot , \cdot , \cdot \rangle \rangle: \calS^{(3)}_{\calX} \rightarrow \bbR^\ast \backslash \bbC^\ast. 
$$
Burger, Iozzi and Wienhard exploited the identifcation between $\calS_{\calX}$ and the real points $(\bfG /\bfQ)(\bbR)$ to extend the Hermitian triple product to the whole $\bfG/\bfQ$. We denote by $A^\ast$ the group $\bbC^\ast \times \bbC^\ast$ endowed with the involution $(\lambda,\mu) \mapsto (\overline{\mu},\overline{\lambda})$ and let $\Delta^\ast$ the image through the diagonal embedding of $\bbC^\ast$. Burger, Iozzi and Wienhard \cite[Corollary 2.17]{BIW07} showed that there exists a rational map 
$$
\langle \langle \cdot , \cdot , \cdot \rangle \rangle_{\bbC}: (\bfG/\bfQ)^3 \rightarrow \Delta^\ast \backslash A^\ast
$$
which fits in the commutative diagram reported below
$$
\xymatrix{
\calS^{(3)}_{\calX} \ar[rr]^{\langle \langle \cdot ,\cdot , \cdot \rangle \rangle} \ar[d]^{i^3} && \bbR^\ast \backslash \bbC^\ast \ar[d]^\Delta \\
(\bfG/\bfQ)^3 \ar[rr]^{\langle \langle \cdot , \cdot , \cdot \rangle \rangle_{\bbC}} && \Delta^\ast \backslash A^\ast ,
}
$$
where $i:\calS_{\calX} \rightarrow \bfG/\bfQ$ identifies $\calS_{\calX}$ with the real points $(\bfG/\bfQ)(\bbR)$ and $\Delta$ is the diagonal embedding. 

The function $\langle \langle \cdot , \cdot , \cdot \rangle \rangle_{\bbC}$ is called \emph{complex Hermitian triple product}. It encodes important information about the structure of the Hermitian symmetric space $\calX$. In fact, consider the (Zariski open) set $\calO_{\eta_0,\eta_1} \subset \bfG/\bfQ$ such that the map 
$$
P_{\eta_0,\eta_1}:\calO_{\eta_0,\eta_1} \rightarrow \bbR \ , P_{\eta_0,\eta_1}(\eta):=\langle \langle \eta_0,\eta_1,\eta \rangle \rangle_{\bbC}
$$
is well-defined. By \cite[Lemma 5.1]{BIW07} we have that $\calX$ is not of tube type if and only if the map $P_{\eta_0,\eta_1}^m$ is not constant for any $m \in \bbN$. 

\section{Cohomology of orbital equivalence relation}\label{sec orbit rel}

In this section we will introduce the main topic of the paper, namely the orbital cohomology. We mainly refer to the reader to the papers by Feldman and Moore \cite{moore1976,feldman:moore}. 

A standard Borel space $(X,\mu)$ is a measure space which is Borel isomorphic to a Polish space (that is a separable completely metrizable space). Consider an equivalence relation $\calR \subset X \times X$ defined on a standard Borel probability space $(X,\mu)$. We are going to suppose that $\calR$ is \emph{countable}, that is the equivalence classes have at most countable cardinality. Feldman and Moore introduced an exotic cohomology theory associated to a countable equivalence relation with coefficients in a Polish Abelian group. Since for our purpose it will be sufficient to look at the cohomology in degree one, we will give a definition \emph{ad hoc}. An important feature of the $1$-cohomology of a countable equivalence relation is that its definition works fine also when the coefficients are a general topological group $G$, not only a Polish Abelian one. 

\begin{defn}\label{def measurable cocycle}
Let $\calR$ be a countable equivalence relation on a standard Borel probability space $(X,\mu)$. Consider a topological group $G$. A \emph{measurable cocycle} for $\calR$ with coefficients in $G$ is a Borel measurable map $c:\calR \rightarrow G$ such that 
\begin{equation}\label{eq measurable cocycle}
c(x,z)=c(y,z)c(x,y) , 
\end{equation}
for almost every pair $(x,y),(y,z),(x,z) \in \calR$. Two measurable cocycles $c_1,c_2$ are \emph{cohomologous} if there exists a measurable function $f:X \rightarrow G$ such that 
\begin{equation}\label{eq cohomology}
f(y)c_2(x,y)=c_1(x,y)f(x) , 
\end{equation}
for almost every $(x,y) \in \calR$. We denote by $\upH^1(\calR;G)$ the $1$-cohomology of $\calR$ with coefficients in $G$, namely the quotient of measurable cocycles modulo cohomology. 
\end{defn}

In this paper we will be interested in the particular case when $\calR$ is an \emph{orbital equivalence relation}. More precisely, let $\Gamma$ be a finitely generated countable group. We consider a measure preserving action of $\Gamma$ on a standard Borel probability space $(X,\mu)$. The orbital equivalence relation $\calR_{\Gamma}$ is defined as follows: two points $x,y \in X$ are related if and only if there exists $\gamma \in \Gamma$ such that $y=\gamma.x$. 

If we define
$$
\Theta: \{ c:\calR_{\Gamma} \rightarrow G \ | \ \textup{$c$ is measurable} \} \rightarrow \{ \sigma:\Gamma \times X \rightarrow G \ | \ \textup{$\sigma$ is measurable} \},
$$
$$
c \mapsto \sigma_c(\gamma,x):=c(x,\gamma.x),
$$
then the image of the set of measurable cocycles corresponds to the set of measurable functions $\sigma:\Gamma \times X \rightarrow G$ such that 
\begin{equation}\label{eq new measurable cocycle}
\sigma(\gamma_1\gamma_2,x)=\sigma(\gamma_1,\gamma_2.x)\sigma(\gamma_2,x),
\end{equation}
for every $\gamma_1,\gamma_2 \in \Gamma$ and almost every $x \in X$. We will call $\sigma$ a \emph{measurable cocycle} for the orbital equivalence relation. As we did for cocycles, we can rewrite the definition of cohomology using the function $\Theta$. In fact, given two measurable cocycles $\sigma_1,\sigma_2:\Gamma \times X \rightarrow G$, we will say that they are \emph{cohomologous} if there exists a measurable function $f:X \rightarrow G$ such that 
\begin{equation}\label{eq new cohomology}
f(\gamma.x)\sigma_2(\gamma,x)=\sigma_1(\gamma,x)f(x) \ ,
\end{equation}
for every $\gamma \in \Gamma$ and almost every $x \in X$. We denote the $1$-cohomology of the orbital equivalence relation $\calR_{\Gamma}$ by $\upH^1(\Gamma \curvearrowright X;G)$ and we call it \emph{orbital cohomology}.

Here we will interested in a more general equivalence relation among cocycles. In fact we will allow different groups as targets. 

\begin{defn}
Let $\sigma_1:\Gamma \times X \rightarrow G_1$ and $\sigma_2:\Gamma \times X \rightarrow G_2$ be two measurable cocycles. We say that they are \emph{equivalent} if there exists an isomorphism $s:G_1 \rightarrow G_2$ such that $s \circ \sigma_1$ is cohomologous to $\sigma_2$. 
\end{defn}

It is worth noticing that a morphism $\Gamma \rightarrow G$ is precisely a measurable cocycle not depending on the space variable in $(X,\mu)$. In fact, cocycles can be viewed as generalized morphisms (they are actually morphisms of groupoids). In this way we obtain a map from the $G$-character variety $\mathrm{Rep}(\Gamma;G)$, that is homomorphisms modulo $G$-conjugation, to the $1$-cohomology $\upH^1(\Gamma \curvearrowright X;G)$. 

The study of the cohomology $\upH^1(\Gamma \curvearrowright X;G)$ may reveal quite hard to approach. For this reason it could be easier to restrict the attention to particular classes of groups, both for $\Gamma$ and $G$. Suppose for instance that $G$ corresponds to (the connected component) of the real points of a real algebraic group $\bfG$. Then we are allowed to give the following:

\begin{defn}\label{def algebraic hull}
Let $\Gamma$ be a finitely generated group and let $(X,\mu)$ be an ergodic standard Borel probability $\Gamma$-space. The \emph{algebraic hull} of a measurable cocycle $\sigma:\Gamma \times X \rightarrow G$ is the $G$-conjugacy class of the smallest algebraic subgroup $\bfL<\bfG$ such that $L=\bfL(\bbR)^\circ$ cointains the image of a cocycle cohomologous to $\sigma$. We say that $\sigma$ is \emph{Zariski dense} if $\bfL=\bfG$. 
\end{defn}

The previous definition works because the group $\bfG$ is algebraic and hence Noetherian \cite[Proposition 9.1]{zimmer:libro}. For the way we defined the algebraic hull, it is canonically attached to the cohomology class of a cocycle. Thus it makes sense to refer to the subset of Zariski dense cohomology classes, denoted by $\upH^1_{ZD}(\Gamma \curvearrowright X;G)$. 

\begin{oss}
Let $\Gamma$ be a finitely generated group and let $(X,\mu)$ and $(Y,\nu)$ be two standard Borel probability $\Gamma$-spaces. Consider a topological group $G$. Given a $\Gamma$-equivariant map $\pi:X \rightarrow Y$ and a measurable cocycle $\sigma:\Gamma \times Y \rightarrow G$, one can consider the \emph{pullback cocycle}, namely
$$
\pi^\ast \sigma:\Gamma \times X \rightarrow G \ , \ \ \pi^\ast\sigma(\gamma,x):=\sigma(\gamma,\pi(x)). 
$$
The pullback construction naturally induces a map at the level of cohomology classes
$$
\pi^\ast:\upH^1(\Gamma \curvearrowright Y;G) \rightarrow \upH^1(\Gamma \curvearrowright X;G).
$$
It can be interesting trying to understand when this map is injective. It is difficult to say something relevant in full generality. However, if one assumes that $G$ is (the real points of) an algebraic group then the injectivity holds on the subset of classes whose algebraic hull is semisimple (see \cite{Fur10} for more details). 
\end{oss}

\section{The pullback of the bounded K\"{a}hler class}

\subsection{Boundary theory for bounded cohomology}\label{sec bound cohom}
The main goal of this section is to introduce the notion of bounded K\"{a}hler class. For more details about the background related to this topic we refer the reader to \cite{monod:libro,burger2:articolo}. 

We start recalling the definition of continuous bounded cohomology. We will not give the usual definition but we will base our approach on boundary theory. Let $G$ be a locally compact group. A \emph{Lebesgue $G$-space} is a standard Borel probability space $(X,\mu)$ where the measure $\mu$ is only quasi-$G$-invariant. A \emph{Banach $G$-module} $E$ is a Banach space endowed with an isometric $G$-action $\pi:G \rightarrow \mathrm{Isom}(E)$. We will always assume that $E$ is the dual of some Banach space. In this way it makes sense to refer to the weak-$^\ast$ Borel structure on $E$.  

\begin{es}\label{ex G module}
Consider a locally compact group $G$ and a Lebesgue $G$-space $(X,\mu)$. The main examples of Banach $G$-modules we will consider in this paper are:
\begin{enumerate}
	\item The field $\bbR$ endowed with its Euclidean structure and trivial $G$-action.
	\item The Banach space $\upL^\infty(X;\bbR)$ of essentially bounded measurable functions with the weak-$^\ast$ structure coming from being the dual of $\upL^1(X;\bbR)$ and isometric $G$-action given by
$$
(g.f)(x):=f(g^{-1}.x) ,
$$
for every $f \in \upL^\infty(X;\bbR)$. With an abuse of notation we referred to an equivalence class in $\upL^\infty$ by fixing a representative. 
\end{enumerate}
\end{es}

Given a Lebesgue $G$-space $(X,\mu)$, we define the module of \emph{bounded weak-$^\ast$ measurable functions} on $X^{\bullet+1}$ as
\begin{align*}
\calB^\infty_{\mathrm{w}^\ast}(X^{\bullet+1};E):=\{ \ f:X^{\bullet+1} \rightarrow E \ | \  &\textup{$f$ is weak-$^\ast$ measurable and} \\
                                                                                                       \|&f\|_\infty:=\sup_{x_0,\ldots,x_\bullet}\| f(x_0,\ldots,x_\bullet)\|_E < \infty\}
\end{align*}
By identifying two bounded measurable functions $f,f' \in \calB^\infty_{\mathrm{w}^\ast}(X^{\bullet+1};E)$ when they coincide almost everywhere, we define the space of \emph{essentially bounded weak-$^\ast$ measurable functions} on $X^{\bullet+1}$, namely
$$
\upL^\infty_{\mathrm{w}^\ast}(X^{\bullet+1};E):=\calB^\infty_{\mathrm{w}^\ast}(X^{\bullet+1};E)/\sim , 
$$
where $f \sim f'$ means that they are identified. With the same abuse of notation of Example \ref{ex G module}, we are going to refer to classes in $\upL^\infty_{\mathrm{w}^\ast}$ by fixing a representative. 

We can endow $\calB^\infty_{\mathrm{w}^\ast}(X^{\bullet+1};E)$ with a structure of Banach $G$-module via the isometric action
$$
(g.f)(x_0,\ldots,x_\bullet):=\pi(g)f(g^{-1}.x_0,\ldots,g^{-1}.x_\bullet) ,
$$
for every $f \in \calB^\infty_{\mathrm{w}^\ast}(X^{\bullet+1};E), g \in G$ and $x_0,\ldots,x_\bullet \in X$. Since the relation $\sim$ is preserved by the previous isometric action, the Banach $G$-module structure on $\calB^\infty_{\mathrm{w}^\ast}(X^{\bullet+1};E)$ naturally descends to a Banach $G$-module structure on $\upL^\infty_{\mathrm{w}^\ast}(X^{\bullet+1};E)$. A function $f \in \calB^\infty_{\mathrm{w}^\ast}(X^{\bullet+1};E)$ (or a class in $\upL^\infty_{\mathrm{w}^\ast}(X^{\bullet+1};E)$) is $G$-\emph{invariant} if $g.f=f$ for every $g \in G$. Similarly, we say that it is \emph{alternating} if
$$
\varepsilon(\tau)f(x_0,\ldots,x_\bullet)=f(x_{\tau(0)},\ldots,x_{\tau(\bullet)}) ,
$$
for every permutation $\tau \in \mathfrak{S}_{\bullet+1}$, where $\varepsilon(\tau)$ is the sign. We denote by $\calB^\infty_{\mathrm{w}^\ast}(X^{\bullet+1};E)^G$ (respectively $\upL^\infty_{\mathrm{w}^\ast}(X^{\bullet+1};E)^G$) the submodule of $G$-invariant vectors and we use the notation $\calB^\infty_{\mathrm{w}^\ast,\mathrm{alt}}(X^{\bullet+1};E)$ (respectively $\upL^\infty_{\mathrm{w}^\ast,\mathrm{alt}}(X^{\bullet+1};E)$) to refer to the subspace of alternating functions. 
 
Together with the \emph{standard homogeneous coboundary operator}
$$
\delta^\bullet:\calB^\infty_{\mathrm{w}^\ast}(X^{\bullet+1};E) \rightarrow \calB^\infty_{\mathrm{w}^\ast}(X^{\bullet+2};E),
$$
$$
(\delta^\bullet f)(x_0,\ldots,x_{\bullet+1}):=\sum_{i=0}^{\bullet+1} (-1)^i f(x_0,\ldots,x_{i-1},x_{i+1},\ldots,x_{\bullet+1}),
$$
we obtain a cochain complex $(\calB^\infty_{\mathrm{w}^\ast}(X^{\bullet+1};E),\delta^\bullet)$. In a similar way, each coboundary operator descends to the quotient, hence we obtain also the cochain complex of essentially bounded functions $(\upL^\infty_{\mathrm{w}^\ast}(X^{\bullet+1};E),\delta^\bullet)$. We will exploit such complex to define the continuous bounded cohomology of $G$. We first need to introduce the notion of boundary.

\begin{defn}\label{def boundary}
Let $G$ be a locally compact group and let $(B,\nu)$ be a Lebesgue $G$-space. We say that $(B,\nu)$ is \emph{amenable} if it admits a $G$-equivariant \emph{mean}, that is a norm-one linear operator 
$$
m:\upL^\infty(G \times B;\bbR) \rightarrow \upL^\infty(B;\bbR) ,
$$
such that $m(\chi_{G \times B})=\chi_B$, $m(f)\geq0$ whenever $f$ is positive and $m(f \cdot \chi_{G \times A})=m(f) \cdot \chi_A$ for any essentially bounded function $f$ and measurable set $A \subset B$. 

An amenable $G$-space $(B,\nu)$ is a $G$-\emph{boundary} (in the sense of Burger and Monod \cite{burger2:articolo}) if any Borel measurable $G$-equivariant function $B \times B \rightarrow \calH$ is essentially constant, where $\calH$ varies in the set of all Hilbert $G$-modules. 
\end{defn}

\begin{es}\label{es boundary}
We give three different examples of $G$-boundary that we will use later.
\begin{enumerate}
\item Let $\bbF_S$ be the free group with symmetric generating set $S$. We want to exhibit a $\bbF_S$-boundary. In this case is sufficient to consider $B=\partial \calT_S$ the boundary of the Cayley graph of $\bbF_S$, namely the set of reduced words on $S$ with infinite length. We endow $B$ with the quasi-invariant measure
$$
\mu_S(C(x))=\frac{1}{2r(2r-1)^{n-1}},
$$
where $x$ is a reduced word of length $n$, $r=|S|$ and $C(x)$ is the cone of infinite reduced words starting with $x$.

\item Consider a finitely generated group $\Gamma$ with symmetric generating set $S$. If $\rho:\bbF_S \rightarrow \Gamma$ is a representation where $N=\ker \rho$ is exactly given by the normal subgroup generated by the relations in $\Gamma$, we can consider the set $\upL^\infty(\partial T_S,\mu_S)^N$ of $N$-invariant essentially bounded functions. By Mackey realization theorem \cite{Mackey} there exists a standard measure space $(B,\nu)$ and a measurable map $\pi:\partial T_S \rightarrow B$ such that $\pi_\ast(\mu_S)=\nu$ and the pullback of $\upL^\infty(B,\nu)$ via $\pi$ is exactly $\upL^\infty(\partial \calT_S,\mu_S)^N$. By \cite[Theorem 2.7]{BF14} we have that $(B,\nu)$ is a $\Gamma$-boundary. 

\item When $\Gamma$ is a lattice in a semisimple Lie group $G$, its $\Gamma$-boundary can be easily realized as the quotient $G/P$, where $P$ is any minimal parabolic subgroup \cite[Theorem 2.3]{BF14}. 
\end{enumerate}
\end{es}

Using the notion of boundary we are finally ready to give the following:

\begin{defn} \label{def bounded cohomology}
Let $G$ be a locally compact group and let $(B,\nu)$ a $G$-boundary. The \emph{continuous bounded cohomology} of $G$ with coefficients in the Banach $G$-module $E$ is the cohomology of the complex
$$
\upH^\bullet_{cb}(G;E) := \upH^\bullet((\upL^\infty_{\mathrm{w}^\ast}(B^{\bullet+1};E)^G,\delta^\bullet)) .
$$
\end{defn}

\begin{oss}\label{oss alt subcomplex}
The same definition remains valid if we restrict ourselves to the subcomplex of essentially bounded alternating functions, namely
$$
\upH^\bullet_{cb}(G;E) \cong \upH^\bullet((\upL^\infty_{\mathrm{w}^\ast,\mathrm{alt}}(B^{\bullet+1};E)^G,\delta^\bullet)) .
$$
\end{oss}

We want to point out that our definition is not the usual one, which relies on another complex defined directly on the group. In fact, one can consider the complex $(\upC_{cb}(G^{\bullet+1};E),\delta^\bullet)$ of $E$-\emph{valued continuous bounded functions} on tuples of $G$, endowed with the same action described for the complex of essentially measurable functions. It is still true that the subcomplex of $G$-invariant vectors computes the continuous bounded cohomology of $G$ \cite[Section 6.1]{monod:libro}. Using such complex, it is also clear that any continuous representation $G \rightarrow H$ induces functorially a map between the bounded cohomologies of $G$ and $H$. This is less clear for our definition based on boundary theory, but our approach will have the advantage to make the computation more explicit. We will make it more clear in the next section.

\begin{es}
When a group $\Gamma$ is discrete (for instance for a finitely generated one or for a lattice), the continuity condition is trivial. Hence we refer simply to the bounded cohomology of $\Gamma$ and we denote it by $\upH^\bullet_b$. 
\begin{enumerate}
	\item Let $\Gamma$ be a discrete countable finitely generated group. Its bounded cohomology $\upH^\bullet_b(\Gamma;E)$ with coefficients in $E$ is given by the cohomology of the complex $(\upL^\infty_{\mathrm{w}^\ast}(B^{\bullet+1};E)^{\Gamma},\delta^\bullet)$, where $B$ is the boundary described in Example \ref{es boundary}(2). 
	\item Suppose that $\Gamma < G$ is a lattice in a semisimple Lie group $G$. If $P<G$ is a minimal parabolic subgroup, the bounded cohomology of $\Gamma$ is given by the cohomology of the complex $(\upL^\infty_{\mathrm{w}^\ast}((G/P)^{\bullet+1};E)^{\Gamma},\delta^\bullet)$ in virtue of Example \ref{es boundary}(3). 
\end{enumerate}
\end{es}

Any $G$-equivariant morphism $\alpha:E \rightarrow F$ between $G$-modules induces a map at the level of continuous bounded cohomology groups
$$
\upH^\bullet_{cb}(\alpha):\upH^\bullet_{cb}(G;E) \rightarrow \upH^\bullet_{cb}(G;F). 
$$
In this paper we will mainly be interested in the map induced by the change of coefficients $\bbR \hookrightarrow \upL^\infty(X;\bbR)$, where $(X,\mu)$ is a Lebesgue $G$-space. 

We conclude this section by spending some words about the complex of bounded measurable functions. Let $(Y,\nu)$ be any Lebesgue $G$-space, not necessarily amenable. Burger and Iozzi \cite[Corollary 2.2]{burger:articolo} proved that there exists a canonical non-trivial map 
$$
\mathfrak{c}^\bullet:\upH^\bullet((\calB^\infty_{\mathrm{w}^\ast}(Y^{\bullet+1};E),\delta^\bullet)^G) \rightarrow \upH^\bullet_{cb}(G;E) ,
$$
and the same holds if we restrict to the alternating subcomplex.

\begin{es} \label{es kahler class}
Let $G$ be a semisimple Hermitian Lie group $G$ with symmetric space $\calX$. If $\calS_{\calX}$ is the Shilov boundary, we know that it is isomorphic to the quotient $G/Q$ (by Section \ref{sec herm space}) and hence it is a Lebesgue $G$-space (since homogeneous quotients admit always a quasi-$G$-invariant measure). The Bergman cocycle $\beta_{\calX}$ is an everywhere defined alternating cocycle that can be considered as an element
$$
\beta_{\calX} \in \calB^\infty_{\mathrm{alt}}(\calS^3_{\calX};\bbR)^G .
$$
By \cite[Proposition 4.3]{BIW07} the image of the class $[\beta_{\calX}]$ under the map 
$$
\mathfrak{c}^2:\upH^2((\calB^\infty_{\mathrm{alt}}(\calS^{\bullet+1}_{\calX};\bbR)^G;\delta^\bullet)) \rightarrow \upH^2_{cb}(G;\bbR)
$$
does not vanish. 
\end{es}

\begin{defn}
Let $G$ be a semisimple Hermitian Lie group with symmetric space $\calX$. We denote by 
$$
k^b_G:=\mathfrak{c}^2[\beta_\calX] \in \upH^2_{cb}(G;\bbR) 
$$
and we call it \emph{bounded K\"{a}hler class}.
\end{defn}

It is well-known \cite{BIW07,Pozzetti} that the bounded K\"{a}hler class is a generator for the second bounded cohomology group. We will exploit this fact in Section \ref{sec maximal} when we are going to speak about maximal cocycles.

\subsection{Pullback along measurable cocycles}\label{sec pullback class}

We are finally ready to introduce  the notion of pullback along a measurable cocycle. We mainly refer to \cite{moraschini:savini,moraschini:savini:2} for a detailed discussion about this topic. 

We will first introduce the pullback using the complex of continuous functions on the group, then we will see how we can implement it in terms of boundaries. 
Let $\Gamma$ be a finitely generated discrete group and let $G$ be a semisimple Hermitian Lie group. Consider a standard Borel probability $\Gamma$-space $(X,\mu)$. Given a measurable cocycle $\sigma:\Gamma \times X \rightarrow G$ we can define
$$
\upC^\bullet_b(\sigma):\upC_{cb}(G^{\bullet+1};\bbR) \rightarrow \upC_{b}(\Gamma^{\bullet+1};\upL^\infty(X;\bbR)) ,
$$
$$
(\upC^\bullet_b(\sigma)(\psi))(\gamma_0,\ldots,\gamma_\bullet)(x):=\psi(\sigma(\gamma_0^{-1},x)^{-1},\ldots,\sigma(\gamma^{-1}_\bullet,x)^{-1}).
$$
The above map is a well-defined cochain map and it induces a map at the level of bounded cohomology \cite[Lemma 2.7]{savini2020}, namely
$$
\upH^\bullet_b(\sigma):\upH^\bullet_{cb}(G;\bbR) \rightarrow \upH^\bullet_b(\Gamma;\upL^\infty(X;\bbR)) , \ \upH^\bullet_b(\sigma)([\psi]):=[\upC^\bullet_b(\sigma)(\psi)] .
$$
Furthermore, when $\sigma_1$ and $\sigma_2$ are cohomologous cocycles, by \cite[Lemma 2.9]{savini2020} we have that
$$
\upH^\bullet_b(\sigma_1)=\upH^\bullet_b(\sigma_2) .
$$

\begin{defn}\label{def pullback kahler}
Let $G$ be a semisimple Hermitian Lie group, let $\Gamma$ be a finitely generated group and let $(X,\mu)$ be a standard Borel probability $\Gamma$-space. Given a measurable cocycle $\sigma:\Gamma \times X \rightarrow G$, we define its \emph{parametrized K\"{a}hler class} as 
$$
\upH^2_b(\sigma)(k^b_G) \in \upH^2_b(\Gamma;\upL^\infty(X;\bbR)) .
$$
\end{defn}

Our next goal is to show how we can implement explicitly the pullback in terms of boundaries. We start with the following 

\begin{defn}\label{def boundary map}
Let $\Gamma$ be a finitely generated group with $\Gamma$-boundary $B$. Consider a standard Borel probability $\Gamma$-space $(X,\mu)$. Given a semisimple Hermitian Lie group $G$, let $(Y,\nu)$ be a Lebesgue $G$-space. A \emph{boundary map} for a measurable cocycle $\sigma:\Gamma \times X \rightarrow G$ is a Borel measurable map 
$$
\phi:B \times X \rightarrow Y, 
$$
which is $\sigma$-equivariant, namely
$$
\phi(\gamma.b,\gamma.x)=\sigma(\gamma.x)\phi(b,x) ,
$$
for all $\gamma \in \Gamma$ and almost every $b \in B, x \in X$. 
\end{defn}

Given a boundary map $\phi:B \times X \rightarrow Y$, the map  
$$
\phi_x:B \rightarrow Y ,
$$
is called $x$-\emph{slice} of $\phi$ and it is Borel measurable by \cite[Chapter VII, Lemma 1.3]{margulis:libro}. The $\sigma$-equivariance of $\phi$ implies that slices change equivariantly as follows:
$$
\phi_{\gamma.x}(b)=\sigma(\gamma,x)\phi_x(\gamma^{-1}b) ,
$$
for all $\gamma \in \Gamma$ and almost every $b \in B,x \in X$. 

Recall $G$ has associated a connected adjoint semisimple real algebraic group $\bfG$ obtained via complexification. Suppose that $Y$ corresponds to the real points of a real algebraic quotient $\bfG/\bfL$, for some real algebraic subgroup $\bfL < \bfG$. We say that the $x$-slice is \emph{Zariski dense} if the Zariski closure of the essential image of $\phi_x$ is the whole $\bfG/\bfL$. 

For our purposes it will be crucial the following:

\begin{thm}{\upshape \cite[Corollary 2.16]{sarti:savini1}}\label{teor boundary map}
Let $\Gamma$ be a finitely generated group with $\Gamma$-boundary $B$ and let $(X,\mu)$ be an ergodic standard Borel probability $\Gamma$-space. Consider a Zariski dense measurable cocycle $\sigma:\Gamma \times X \rightarrow G$ into a semisimple Hermitian Lie group $G$. Then there exists a boundary map $\phi:B \times X \rightarrow G/Q$, where $G/Q$ is the algebraic realization of the Shilov boundary associated to $G$. Moreover, almost every slice is Zariski dense and preserves transversality, that is $\phi(b_0,x),\phi(b_1,x)$ are transverse whenever $b_0,b_1$ are so. 
\end{thm}

We want to use a boundary map to realize the pullback in bounded cohomology. A delicate point already observed by Burger and Iozzi \cite{burger:articolo} is that a priori the slices of a boundary map does not need to preserve the measure classes involved. To overcome such problem, we will consider directly the space of bounded measurable functions. Given a boundary map $\phi:B \times X \rightarrow Y$ for a measurable cocycle $\sigma:\Gamma \times X \rightarrow G$, we can define
$$
\upC^\bullet(\phi):\calB^\infty(Y^{\bullet+1};\bbR)^G \rightarrow \upL^\infty_{\mathrm{w}^\ast}(B^{\bullet+1};\upL^\infty(X;\bbR))^\Gamma
$$
$$
(\upC^\bullet(\phi)(\psi))(b_0,\ldots,b_\bullet)(x):=\psi(\phi(b_0,x),\ldots,\phi(b_\bullet,x)) ,
$$
where we tacitly postcomposed with the projection on the essentially bounded functions on $B$. By \cite[Lemma 4.2]{moraschini:savini} the map $\upC^\bullet(\phi)$ is a norm non-increasing cochain map which induces 
$$
\upH^\bullet(\phi):\upH^\bullet(\calB^\infty(Y^{\bullet+1};\bbR)^G,\delta^\bullet) \rightarrow \upH^\bullet_b(\Gamma;\upL^\infty(X;\bbR)) , \ \upH^\bullet(\phi)([\psi]):=[\upC^\bullet(\phi)(\psi)] . 
$$
By applying \cite[Proposition 2.1]{burger:articolo} we obtain the following commutative diagram
\begin{equation} \label{diagram pullback}
\xymatrix{
\upH^\bullet(\calB^\infty(Y^{\bullet+1};\bbR)^G,\delta^\bullet) \ar[rr]^{\mathfrak{c}^\bullet} \ar[d]^{\upH^\bullet(\phi)} && \upH^\bullet(G;\bbR) \ar[dll]^{\upH^\bullet_b(\sigma)}\\
\upH^\bullet_b(\Gamma;\upL^\infty(X;\bbR)) . 
}
\end{equation}

\begin{es}\label{ es boundary kahler class}
Let $\Gamma$ be a finitely generated group and let $G$ be a semisimple Hermitian Lie group with symmetric space $\calX$. Consider a Zariski dense measurable cocycle $\sigma:\Gamma \times X \rightarrow G$, where $(X,\mu)$ is an ergodic standard Borel probability $\Gamma$-space. By Theorem \ref{teor boundary map} there exists a boundary map $\phi:B \times X \rightarrow \calS_{\calX}$ whose slices are Zariski dense and preserve transversality. By Example \ref{es kahler class} we have that $\mathfrak{c}^2[\beta_{\calX}]$ is the bounded K\"{a}hler class $k^b_G$. By Definition \ref{def pullback kahler} we know that $\upH^2_b(\sigma)(k^b_G)$ is the parametrized K\"{a}hler class. Thus Diagram \ref{diagram pullback} shows that a canonical non-trivial representative of the parametrized K\"{a}hler class is given by $\upC^2(\phi)(\beta_{\calX})$, namely
$$
\upC^2(\phi)(\beta_{\calX})(b_0,b_1,b_2)(x):=\beta_{\calX}(\phi(b_0,x),\phi(b_1,x),\phi_2(b_2,x)) .
$$
\end{es}

\section{Main results} \label{sec main result}

\subsection{Rigidity for Zariski dense cocycles}\label{sec rigidity} 

Let $\Gamma$ be a finitely generated group and let $(X,\mu)$ be an ergodic standard Borel probability $\Gamma$-space. Consider a simple Hermitian Lie group $G$ not of tube type. 
In this section we want to show how the parametrized K\"{a}hler class encodes all the information associated to a Zariski dense $G$-valued measurable cocycle. More precisely, we will see that we can embed the Zariski dense $G$-orbital cohomology in the second bounded cohomology group of $\Gamma$ with $\upL^\infty(X;\bbR)$-coefficients. 

To see this we start recalling the following more general result. 

\begin{thm}{\upshape \cite[Theorem 2]{sarti:savini3}} \label{ teor parametrized inequivalent}
Let $\sigma_i:\Gamma \times X \rightarrow G_i$, for $i=1,\ldots,n$, be a measurable cocycles into a simple Hermitian Lie group $G_i$ not of tube type. Suppose that the cocycles are Zariski dense and pairwise inequivalent. Then the subset 
$$
\{ \upH^2_b(\sigma_i)(k^b_{G_i}) \}_{i=1,\ldots,n} \subset \upH^2_b(\Gamma;\upL^\infty(X;\bbR))
$$
is linearly independent over $\upL^\infty(X;\mathbb{Z})$.
\end{thm}

\begin{proof}[Sketch of the proof]

By Theorem \ref{teor boundary map} there exists a boundary map $\phi_i:B \times X \rightarrow \calS_i$, where $B$ is a $\Gamma$-boundary and $\calS_i$ is the Shilov boundary for $G_i$. Notice that by \cite[Corollary 2.6]{MonShal0} there are no coboundaries in degree $2$. Thanks to Example \ref{ es boundary kahler class} any trivial combination
$$
\sum_{i=1}^n m_i \upH^2_b(\sigma_i)(k^b_{G_i})=0 ,
$$
where $m_i \in \upL^\infty(X;\mathbb{Z})$, boils down to the following equation
\begin{equation}\label{eq linear combination}
\sum_{i=1}^n m_i(x) \beta_i(\phi_i(b_0,x),\phi_i(b_1,x),\phi_i(b_2,x))=0 , 
\end{equation}
for almost every $b_0,b_1,b_2 \in B$ and $x \in X$. Here $\beta_i$ is the Bergman cocycle on the Shilov boundary $\calS_i$, for $i=1,\ldots,n$. Using Equation \eqref{eq Bergman triple product} we can rewrite the previous linear combination in terms of complex Hermitian triple products, namely
$$
\prod_{i=1}^n \langle \langle \phi_i(b_0,x),\phi_i(b_1,x),\phi_i(b_2,x) \rangle \rangle_{\bbC}^{m_i(x)}=1 ,
$$
for almost every $b_0,b_1,b_2 \in B, x \in X$. 

By the transitivity of $G_i$ on transverse pairs in $\calS_i$, one can find a cocycle $\widetilde{\sigma}_i$ cohomologous to $\sigma$ with boundary map $\widetilde{\phi}_i:B \times X \rightarrow \calS_i$, such that the images $\widetilde{\phi}_i(b_0,x)=\eta_i$ and $\widetilde{\phi_i}(b_1,x)=\zeta_i$ do not depend on $x \in X$ and furthermore it holds that
\begin{equation}\label{eq product Hermitian products}
\prod_{i=1}^n \langle \langle \eta_i , \zeta_i ,\widetilde{\phi}_i(b_2,x) \rangle \rangle_{\bbC}^{m_i(x)}=1 ,
\end{equation}
for almost every $b_2 \in B, x \in X$. 

If we consider the product cocycle
$$
\widetilde{\sigma}:\Gamma \times X \rightarrow \prod_{i=1}^n G_i , \ (\gamma,x) \mapsto (\widetilde{\sigma}_i(\gamma,x))_{i=1,\ldots,n}
$$
with boundary map 
$$
\widetilde{\phi}:B \times X \rightarrow \prod_{i=1}^n \calS_i , \ (b,x) \mapsto (\widetilde{\phi}_i(b,x))_{i=1,\ldots,n} , 
$$
Equation \eqref{eq product Hermitian products} and the fact that each $G_i$ is not of tube type imply that almost every $x$-slice of $\widetilde{\phi}$ is not Zariski dense, since the Zariski closure of the essential image of almost each slice is contained in the proper Zariski closed set
$$
\{ (\omega_1,\ldots,\omega_n) \in \prod_{i=1}^n \calO_{\eta_i,\zeta_i} \ | \ \prod_{i=1}^n P_i^{m_i(x)}(\omega_i)=1 \}. 
$$
Here $\calO_{\eta_i,\zeta_i}$ is the Zariski open set defined at the end of Section \ref{sec herm space}. By Theorem \ref{teor boundary map} the algebraic hull $\bfL$ of $\widetilde{\sigma}$ must be a proper subgroup of the product $\prod_{i=1}^n \bfG_i$, where $\bfG_i$ is the connected adjoint simple algebraic group obtained by complexifying $G_i$, for $i=1,\ldots,n$. Since $\bfL$ surjects on each $\bfG_i$ via projections and $\bfG_i$ are simple, there must exist at least one $\bbR$-isomorphism $s:\bfG_i \rightarrow \bfG_j$ for $i\neq j \in \{1,\ldots,n\}$. This is a contradiction to the inequivalence of the $\sigma_i$'s. 

\end{proof}

Using Theorem \ref{ teor parametrized inequivalent} one can show the following:

\begin{thm}{\upshape \cite[Theorem 1]{sarti:savini3}}\label{kx injection}
Let $\Gamma$ be a finitely generated group and $(X,\mu)$ be an ergodic standard Borel probability $\Gamma$-space. Consider a simple Hermitian Lie group $G$. 
The map 
$$
K_X:\upH^1_{ZD}(\Gamma \curvearrowright X;G) \rightarrow \upH^2_b(\Gamma;\upL^\infty(X;\bbR)) , \ \ K_X([\sigma]):=\upH^2_b(\sigma) (k^b_{G})
$$
is an injection whose image avoids the trivial class. As a consequence the parametrized K\"{a}hler class is a complete invariant for the orbital cohomology class of a Zariski dense cocycle $\sigma$. 
\end{thm}

\begin{proof}[Sketch of the proof]
Let $\sigma_1,\sigma_2:\Gamma \times X \rightarrow G$ be two Zariski dense cocycles. We need to show that if $\upH^2_b(\sigma_1)=\upH^2_b(\sigma_2)$, then $\sigma_1$ and $\sigma_2$ are cohomologous. By Theorem \ref{ teor parametrized inequivalent} we have that $\sigma_1$ and $\sigma_2$ are equivalent, thus there exists a $\bbR$-isomorphisms $s:\bfG \rightarrow \bfG$ of the connected adjoint simple algebraic group $\bfG$ associated to $G$, such that $s \circ \sigma_1$ is cohomologous to $\sigma_2$. Since the pullback is equivariant with respect to the sign of $s$, we have that
$$
0=\upH^2_b(\sigma_1)-\upH^2_b(\sigma_2)=\upH^2_b(\sigma_1)-\varepsilon(s)\upH^2_b(\sigma_1)=(1-\varepsilon(s))\upH^2_b(\sigma_1) . 
$$

Again Theorem \ref{ teor parametrized inequivalent} implies that $\upH^2_b(\sigma_1)$ is not trivial, thus $\varepsilon(s)=1$ and the statement follows. 
\end{proof}

The previous theorem has important consequences on the computation of the orbital cohomology when $\Gamma$ is either a higher rank lattice or it is a lattice in a product. 

\begin{prop}{\upshape \cite[Proposition 4.1]{sarti:savini3}}\label{prop higher rank}
Let $\Gamma < H=\bfH(\bbR)^\circ$ be a lattice, where $\bfH$ is a connected, simply connected, almost simple $\bbR$-group of real rank at least $2$. Let $(X,\mu)$ be an ergodic standard Borel probability $\Gamma$-space and let $G$ be a simple Hermitian Lie group. If $\upH^2_b(\Gamma;\bbR)\cong 0$ then 
$$
|\upH^1_{ZD}(\Gamma \curvearrowright X;G)|=0 .
$$
\end{prop}

\begin{proof}
Thanks to Theorem \ref{kx injection} we have an injection
$$
K_X:\upH^1_{ZD}(\Gamma \curvearrowright X;G) \rightarrow \upH^2_b(\Gamma;\upL^\infty(X;\bbR)) 
$$
whose image avoids the trivial class. Since $\upL^\infty(X;\bbR)$ is semiseparable as Banach $G$-module, by \cite[Corollary 1.6]{Mon10} we have the following chain of isomorphisms
$$
\upH^2_b(\Gamma;\upL^\infty(X;\bbR)) \cong \upH^2_b(\Gamma;\upL^\infty(X;\bbR)^\Gamma) \cong \upH^2_b(\Gamma;\bbR),
$$
where the last isomorphism is due to the ergodicity of $(X,\mu)$. By assumption the statement now follows. 
\end{proof}

We refer either \cite{BM1,burger2:articolo} to see when the hypothesis $\upH^2_b(\Gamma;\bbR) \cong 0$ is satisfied. In virtue of Proposition \ref{prop higher rank} we have a vanishing result for the Zariski dense orbital cohomology. Such an explicit result is usually difficult to obtain and this is exactly why we should understand the importance of having a rigidity result as Theorem \ref{kx injection}. 

We conclude with the case of products. Recall that a lattice $\Gamma < H:=H_1 \times \ldots \times H_n$ in a product of locally compact second countable groups is \emph{irreducible} if it projects densely on each $H_i$. Additionally, we say that $H$ acts \emph{irreducibly} on a standard Borel probability space $(X,\mu)$ if each subgroup obtained by omitting one factor of $H$ acts ergodically on $X$. 

\begin{prop}{\upshape \cite[Proposition 4.4]{sarti:savini3}}\label{prop products}
Consider $n \geq 2$ and consider an irreducible lattice $\Gamma < H:=H_1 \times \ldots \times H_n$ in a product of locally compact second countable groups such that $\upH^2_{cb}(H_i;\bbR)=0$ for $i=1,\ldots,n$. Let $(X,\mu)$ be a standard Borel $H$-irreducible probability space and consider a simple Hermitian Lie group $G$. Then 
$$
|\upH^1_{ZD}(\Gamma \curvearrowright X;G)|=0 .
$$
\end{prop}

\begin{proof}
By \cite[Corollary 9]{Mon10} the inclusion 
$$
\upL^\infty(X;\bbR) \rightarrow \upL^2(X;\bbR) 
$$
induces an injection in bounded cohomology. Precomposing with $K_X$, we obtain an injection 
$$
\upH^1_{ZD}(\Gamma \curvearrowright X;G) \rightarrow \upH^2_b(\Gamma;\upL^2(X;\bbR))
$$
which avoids the trivial class. If we set
$$
H'_i:=\prod_{j \neq i} H_j ,
$$
by \cite[Theorem 16]{burger2:articolo} we have that 
$$
\upH^2_b(\Gamma;\upL^2(X;\bbR)) \cong \bigoplus_{i=1}^n \upH^2_b(H_i;\upL^2(X;\bbR)^{H'_i}) \cong \upH^2_b(H_i;\bbR) 
$$
and the statement follows.

\end{proof}

\subsection{Maximal measurable cocycles}\label{sec maximal}

So far we have seen the theory of pullback along a Zariski dense cocycle $\Gamma \times X \rightarrow G$ with values in a simple Hermitian Lie group in full generality. Our next goal is to assume some more restrictive conditions on both $\Gamma$ and $G$ and to introduce a new family of measurable cocycles, namely maximal ones. We mainly refer to \cite{sarti:savini1} for more details about this topic. 

We set $G_{p,q}:=\pupq$. Consider a lattice $\Gamma < G_{n,1}$, with $n \geq 2$, and a standard Borel probability $\Gamma$-space $(X,\mu)$. Since the measure $\mu$ is finite, the change of coefficients 
$$
\upH^2_b(\Gamma;\bbR) \rightarrow \upH^2_b(\Gamma;\upL^\infty(X;\bbR)) . 
$$
admits a left inverse induced by integration along $X$. More precisely, if we consider 
$$
\upI_X^\bullet:\upC_b(\Gamma^{\bullet+1};\upL^\infty(X;\bbR)) \rightarrow \upC_b(\Gamma^{\bullet+1};\bbR) ,
$$
$$
\upI_X^\bullet(\psi)(\gamma_0,\ldots,\gamma_\bullet):=\int_X \psi(\gamma_0,\ldots,\gamma_\bullet)(x) d\mu(x) , 
$$
we have that $\upI_X$ is a norm non-increasing cochain map which induces a map at a the level of cohomology groups
$$
\upI^\bullet_X:\upH^\bullet_b(\Gamma;\upL^\infty(X;\bbR)) \rightarrow \upH^\bullet_b(\Gamma;\bbR) .
$$ 
Since $\Gamma$ is a lattice (and hence the quotient $\Gamma \backslash G_{n,1}$ has finite Haar measure), also the restriction map 
$$
\upH^2_{cb}(G_{n,1};\bbR) \rightarrow \upH^2_b(\Gamma;\bbR) 
$$
admits an inverse, this time a right one. If we define the \emph{transfer map} as
$$
\upT_b^\bullet:\upC_b(\Gamma^{\bullet+1};\bbR) \rightarrow \upC_{cb}(G_{n,1}^{\bullet+1};\bbR) ,
$$
$$
(\upT_b\psi)(g_0,\ldots,g_\bullet):=\int_{\Gamma \backslash \puno} \psi(\overline{g}g_0,\ldots,\overline{g}g_\bullet)d\mu_{\Gamma \backslash \puno}(\overline{g}),
$$
we obtain a cochain map inducing the \emph{cohomological transfer map}
$$
\upT^\bullet_b:\upH^\bullet_b(\Gamma;\bbR) \rightarrow \upH^\bullet_{cb}(G_{n,1};\bbR) . 
$$

Given a measurable cocycle $\sigma:\Gamma \times X \rightarrow G_{p,q}$, with $1 \leq p \leq q$, we can consider the image of  the K\"{a}hler class $k^b_{p,q} \in \upH^2_b(G_{p,q};\bbR)$ through the following composition  
$$
(\upT^2_b \circ \upI_X^2 \circ \upH^2_b(\sigma))(k^b_{p,q}) \in \upH^2_b(G_{n,1};\bbR) .
$$
Since the latter group is one dimensional and generated by the K\"{a}hler class $k^b_{n,1}$, we are allowed to give the following:

\begin{defn}\label{def toledo invariant}
The \emph{Toledo invariant} associated to a measurable cocycle $\sigma:\Gamma \times X \rightarrow G_{p,q}$ is the real number $\mathrm{t}_b(\sigma)$ which satisfies the following identity
\begin{equation}\label{eq toledo invariant}
(\upT^2_b \circ \upI^2_b \circ \upH^2_b(\sigma))(k^b_{p,q})=\mathrm{t}_b(\sigma)k^b_{n,1} .
\end{equation}
\end{defn}

The Toledo invariant of a measurable cocycle $\sigma:\Gamma \times X \rightarrow G_{p,q}$ is invariant along the orbital cohomology class of $\sigma$. As a consequence it induces a function 
$$
\mathrm{t}_b:\upH^1(\Gamma \curvearrowright X;G_{p,q}) \rightarrow \bbR . 
$$
The image of the previous function is contained in a bounded interval, in fact the Toledo invariant satisfies
$$
|\mathrm{t}_b(\sigma)|\leq \mathrm{rk}(G_{p,q})=\min \{p,q\}=p 
$$ 
and those cocycles which attain the extremal values are called \emph{maximal cocycles}. This allows to define the maximal orbital cohomology $\upH^1_{\max}(\Gamma \curvearrowright X;G_{p,q})$ as the preimage along the Toledo function of the extremal values. Additionally, we denote by $\upH^1_{\max,ZD}(\Gamma \curvearrowright X;G_{p,q})$ the subset of maximal Zariski dense classes. 

\begin{thm}{\upshape \cite[Theorem 2]{sarti:savini1}}\label{teor superrigidity}
Let $\Gamma \leq G_{n,1}$, with $n \geq 2$, be a lattice and let $(X,\mu)$ be an ergodic standard Borel probability $\Gamma$-space. Any maximal Zariski dense cocycle in $G_{p,q}$, where $1 \leq p \leq q$, is cohomologous to a representation $\Gamma \rightarrow G_{p,q}$ with the same properties.
\end{thm}

\begin{proof}[Sketch of the proof]
We assume that the Zariski dense cocycle $\sigma:\Gamma \times X \rightarrow G_{p,q}$ is maximal. Up to changing it sign by composing it with an antiholomorphic isomorphism, we can suppose that $\sigma$ is positively maximal. Additionally, since $\sigma$ is Zariski dense, we can apply Theorem \ref{teor boundary map} to get a boundary map $\phi:\partial_\infty \bbH^n_{\bbC} \times X \rightarrow \calS_{p,q}$, where $\calS_{p,q}$ is the Shilov boundary associated to $G_{p,q}$. 

Since in degree $2$ there are no coboundaries \cite[Corollary 2.6]{MonShal0}, we can rewrite Equation \eqref{eq toledo invariant} as follows
\begin{align}\label{eq toledo boundary}
\int_{\Gamma \backslash G_{n,1}} \int_X &\beta_{p,q}(\phi(\overline{g}b_0,x),\phi(\overline{g}b_1,x),\phi(\overline{g}b_2,x)) d\mu(x)d\mu_{\Gamma \backslash G_{n,1}}(\overline{g})\\
=\mathrm{t}_b(\sigma)&\beta_{n,1}(b_0,b_1,b_2) , \nonumber
\end{align}
for almost every $b_0,b_1,b_2 \in B$ and $x \in X$. The equation can be actually extended to every triple $b_0,b_1,b_2$ of points that are pairwise distinct. Since $\phi_x$ is Zariski dense \cite[Proposition 4.4]{sarti:savini1} for almost every $x \in X$, Equation \eqref{eq toledo boundary} and \cite[Theorem 1.6]{Pozzetti} imply that $\phi_x$ is the restriction of a rational map for almost every $x \in X$ (both $\partial_\infty \bbH^n_{\bbC}$ and $\calS_{p,q}$ are the real points of some real algebraic variety). Thanks to this rationality condition, one can find a measurable map $f:X \rightarrow G_{p,q}$ such that 
\begin{equation} \label{eq separation variables}
\phi(b,x)=f(x)\phi_0(b) ,
\end{equation}
where $\phi_0:\partial_\infty \bbH^n_{\bbC} \rightarrow \calS_{p,q}$ is still rational and Zariski dense.

By setting 
$$
\widetilde{\sigma}:\Gamma \times X \rightarrow G_{p,q}, \ \widetilde{\sigma}(\gamma,x):=f(\gamma.x)^{-1}\sigma(\gamma,x)f(x) ,
$$
one can see that the separation of variables contained in Equation \eqref{eq separation variables} implies that $\widetilde{\sigma}$ does not depend on $x \in X$ and hence it is the desired representation $\Gamma \rightarrow G_{p,q}$. 
\end{proof}

\begin{cor}{\upshape \cite[Proposition 3]{sarti:savini1}}\label{cor no cocycles}
Let $\Gamma \leq G_{n,1}$, with $n \geq 2$, be a lattice and let $(X,\mu)$ be an ergodic standard Borel probability $\Gamma$-space. There is no maximal Zariski dense cocycle $\Gamma \times X \rightarrow G_{p,q}$ when $1 < p < q$. Equivalently
$$
|\upH^1_{\max,ZD}(\Gamma \curvearrowright X;G_{p,q})|=0 . 
$$
\end{cor}

\begin{proof}
Let $\sigma:\Gamma \times X \rightarrow G_{p,q}$ be a maximal Zariski dense cocycle. By Theorem \ref{teor superrigidity} we have a maximal Zariski dense representation $\Gamma \rightarrow G_{p,q}$ contained in the orbital cohomology class of $\sigma$. By \cite[Corollary 1.2]{Pozzetti} there are no maximal Zariski dense representation when $1 < p < q$. 
\end{proof}

\bibliographystyle{amsalpha}
\bibliography{biblionote}
\end{document}